\documentclass[11pt,reqno,a4paper]{amsart}
\usepackage{euscript}

\usepackage{hyperref}
\RequirePackage[dvipsnames]{xcolor} % [dvipsnames]
\definecolor{halfgray}{gray}{0.55} 
\definecolor{webgreen}{rgb}{0,0.5,0}
\definecolor{webbrown}{rgb}{.6,0,0} \hypersetup{%
  colorlinks=true, linktocpage=true, pdfstartpage=3,
  pdfstartview=FitV,%
  breaklinks=true, pdfpagemode=UseNone, pageanchor=true,
  pdfpagemode=UseOutlines,%
  plainpages=false, bookmarksnumbered, bookmarksopen=true,
  bookmarksopenlevel=1,%
  hypertexnames=true,
  pdfhighlight=/O,%hyperfootnotes=true,%nesting=true,%frenchlinks,%
  urlcolor=webbrown, linkcolor=RoyalBlue,
  citecolor=webgreen, %pagecolor=RoyalBlue,%
  pdftitle={},%
  pdfauthor={},%
  pdfsubject={2000 MAthematical Subject Classification: Primary:},%
  pdfkeywords={},%
  pdfcreator={pdfLaTeX},%
  pdfproducer={LaTeX with hyperref}%
}

\newtheorem{theorem}{Theorem}
\newtheorem{proposition}{Proposition}
\newtheorem{corollary}{Corollary}
\newtheorem{lemma}{Lemma}
\newtheorem{remark}{Remark}
\newtheorem{example}{Example}
\renewcommand{\epsilon}{\varepsilon}
\renewcommand{\phi}{\varphi}

\def\Z{\mathbb{Z}}
\def\R{\mathbb{R}}
\def\cA{\EuScript{A}}

\def\Id{\text{\rm Id}}
\def\A{\mathbf{A}}
\begin{document}

\title{Shadowing for nonautonomous dynamics}

\begin{abstract}
We prove that whenever a sequence of invertible and bounded operators $(A_m)_{m\in \Z}$ acting on a Banach space $X$ admits an exponential dichotomy and a sequence of differentiable maps $f_m \colon X\to X$, $m\in \Z$, has bounded and H\"{o}lder derivatives, the nonautonomous dynamics given by $x_{m+1}=A_mx_m+f_m(x_m)$, $m\in \Z$ has various shadowing properties. 
Hence, we extend recent results of  Bernardes Jr. et al. in several directions. 
As a nontrivial  application of our results, we give a new proof of the nonautonomous Grobman-Hartman theorem.
\end{abstract}

\author{Lucas Backes}
\address{\noindent Departamento de Matem\'atica, Universidade Federal do Rio Grande do Sul, Av. Bento Gon\c{c}alves 9500, CEP 91509-900, Porto Alegre, RS, Brazil.}
\email{lhbackes@impa.br} 

\author{Davor Dragi\v cevi\'c}
\address{Department of Mathematics, University of Rijeka, Croatia}
\email{ddragicevic@math.uniri.hr}

\keywords{Shadowing, Nonautonomus systems, Exponential dichotomies, Nonlinear perturbations}
\subjclass[2010]{Primary: 37C50, 34D09; Secondary: 34D10.}
\maketitle

\section{Introduction}
Given a dynamical system $f:M\to M$ acting on an arbitrary  metric space $(M,d)$, a \emph{$\delta$-pseudotrajectory} for $f$ is any sequence of points $(y_n)_{n\in \Z}\subset M$ satisfying
\begin{displaymath}
d(y_{n+1},f(y_n))\le \delta \text{ for every } n\in\Z. 
\end{displaymath}
An important question that naturally arises  is to understand whether a pseudotrajectory can be approximated by a real trajectory for $f$. More precisely,   we ask if there exists a point $x\in M$ such that $d(y_n,f^n(x))$ is small for every $n\in \Z$. If this happens,  we say that $x$ \emph{shadows} the pseudotrajectory $(y_n)_{n\in \Z}$. Furthermore, if  every pseudotrajectory has a shadowing point we say that the system $(f,M)$ has the \emph{shadowing property} (see~\cite{Pal00, Pil99} for more precise definitions).

We emphasize that there is a clear motivation for studying shadowing properties. For instance, a pseudotrajectory can be viewed as a result of measurements in a real system (that are usually subjected to round-off error or noise). Then, to study the  shadowing problem is to try to understand the relationship between the real trajectories and the approximate trajectories obtained by such measurements. If the system has the shadowing property we  have  that the numerically obtained trajectories actually reflect the real behavior of trajectories of $f$. From a more theoretical perspective, the shadowing property has proved to be very fruitful when dealing with problems concerned with topological stability and construction of symbolic dynamics \cite{Bow75}.

It is well known that many classes of dynamical systems exhibit  shadowing property.  Indeed, Bowen~\cite{Bow75} and Anosov~\cite{Ano70} proved that \emph{uniformly hyperbolic} dynamical systems  with both discrete and continuous time have  shadowing property. Furthermore, Katok~\cite{Katok} proved that \emph{nonuniformly hyperbolic} dynamical systems also
posses certain type of shadowing property.  More recently, Pilyugin~\cite{Pil97, Pil99} proved that structurally stable systems exhibit shadowing property.  Finally, we recall that it has been noticed that partially hyperbolic dynamical systems have no shadowing property~\cite{BDT}. 

The original proofs that uniformly hyperbolic dynamical systems have shadowing property (due to Bowen and Anosov)  rely on the existence of invariant stable and unstable manifolds. Later, Palmer~\cite{Palmer2} and independently Mayer and Sell~\cite{MS} gave quite  ingenious and very  simple analytic proofs of shadowing lemma that don't use the invariant manifold theory. 
Their approach also inspired versions of the shadowing lemma that deal with maps on Banach spaces (see~\cite{CLP, Henry}). In addition, analytic proof of the shadowing lemma for nonuniformly hyperbolic dynamics was given in~\cite{DS}. 
 Finally, in the recent paper~\cite{BCDMP} the authors have developed shadowing theory for linear operators on an arbitrary Banach space. 
 We refer to~\cite{Pal00, Pil99} for more discussion and further references related to shadowing theory. 

We emphasize that all the above mentioned results deal with \emph{autonomous} dynamics and the main objective of the present paper is to develop shadowing theory for \emph{nonautonomous} systems acting on an arbitrary Banach space. More precisely, starting with a linear dynamics
\begin{equation}\label{1021}
x_{m+1}=A_m x_m \quad m\in \Z, 
\end{equation}
where the sequence $(A_m)_{m\in \Z}$ admits an exponential dichotomy (which represents a version of the notion of hyperbolicity for time-varying dynamics), we prove that  a small \emph{nonlinear} perturbation of~\eqref{1021} has the shadowing property.  In particular, results from~\cite{BCDMP} correspond to a special case when $(A_m)_{m\in \Z}$  in~\eqref{1021} is a constant sequence of operators and when linear dynamics is not perturbed.
Moreover, we propose a unified approach (inspired by~\cite{DD}) that allows us to measure
the error in the notion of a pseudotrajectory as well as its deviation from a trajectory in  a variety of ways that include the concept of  standard shadowing (as well as other concepts of shadowing studied in~\cite{BCDMP} and elsewhere) as a very particular case. We note that our methods are inspired by the above mentioned analytic proofs of the shadowing lemma.

Finally, as a nontrivial application of our results, we give a simple proof of the nonautonomous version of the Grobman-Hartman theorem in the discrete-time setting. We emphasize that the first version of this result goes back to Palmer~\cite{Palmer} who considered finite-dimensional dynamics with continuous time. More recent versions are due to Barreira and Valls~\cite{BarVal-DCDS07,
BV1} (see also~\cite{BDV1}) who deal with discrete-time dynamics on an arbitrary Banach space that admits a nonuniform exponential dichotomy.  Although our version of the nonautonomous Grobman-Hartman theorem works under more restrictive assumptions that those we just mentioned (see Remark~\ref{10:53} for a detailed discussion), our proof differs from theirs and 
is inspired by the classical construction of the conjugacy between Anosov diffeomorphism and its small perturbation (see~\cite{KH}). 

The paper is organized as follows. In Section~\ref{Prel} we introduce the class of sequence spaces that will be used to present a general framework that will unify various types of shadowing. We also recall the classical notion of an exponential dichotomy. Section~\ref{MR} contains main results of our paper,  while in Section~\ref{A} we present some applications.

\section{Preliminaries}\label{Prel}
\subsection{Banach sequence spaces}
In this subsection we present some basic
definitions and properties from the theory of Banach sequence spaces. The material is taken from~\cite{DD, Sasu} where the reader can also find more details. 

Let $\mathcal{S}(\Z)$ be the set of all sequences $\mathbf{s}=(s_n)_{n\in \Z}$ of real numbers. We say that a linear subspace $B\subset \mathcal{S}(\Z)$ is a \emph{normed sequence space} (over $\Z$) if there exists a norm $\lVert \cdot \rVert_B \colon B \to \R_0^+$ such that if $\mathbf{s}'=(s_n')_{n\in \Z}\in B$ and $\lvert s_n\rvert \le \lvert s_n'\rvert$ for $n\in \Z$, then $\mathbf{s}=(s_n)_{n\in \Z}\in B$ and $\lVert \mathbf{s}\rVert_B \le \lVert \mathbf{s}'\rVert_B$. If in addition $(B, \lVert \cdot \rVert_B)$ is complete, we say that $B$ is a \emph{Banach sequence space}.

Let $B$ be a Banach sequence space over $\Z$. We say that $B$ is \emph{admissible} if:
\begin{enumerate}
\item
$\chi_{\{n\}} \in B$ and $\lVert \chi_{\{n\}}\rVert_B >0$ for $n\in \Z$, where $\chi_A$ denotes the characteristic function of the set $A\subset \Z$;
\item
for each $\mathbf{s}=(s_n)_{n\in \Z}\in B$ and $m\in \Z$, the sequence $\mathbf{s}^m=(s_n^m)_{n\in \Z}$ defined by $s_n^m=s_{n+m}$ belongs to $B$ and  $\lVert \mathbf{s}^m \rVert_B = \lVert \mathbf{s}\rVert_B$.
\end{enumerate}
Note that it follows from the definition that for each admissible Banach space $B$ over $\Z$, we have that $\lVert \chi_{\{n\}}\rVert_B=\lVert \chi_{\{0\}}\rVert_B$ for each $n\in \Z$. Throughout this paper we will assume for the sake of simplicity  that $\lVert \chi_{\{0\}}\rVert_B=1$.

We recall some  explicit examples of admissible  Banach sequence spaces over $\Z$ (see~\cite{DD,Sasu}).

\begin{example}\label{ex1}
The set $l^\infty =\{ \mathbf{s}=(s_n)_{n\in \Z} \in \mathcal{S} (\Z): \sup_{n\in \Z} \lvert s_n \rvert < \infty \}$ is a Banach sequence space when equipped with the norm $\lVert \mathbf{s} \rVert =\sup_{n\in \Z} \lvert s_n \rvert$.
\end{example}

\begin{example}\label{ex2}
The set $c_0=\{ \mathbf{s}=(s_n)_{n\in \Z} \in \mathcal{S} (\Z): \lim_{\lvert n\rvert \to \infty} \lvert s_n\rvert=0\}$ is a Banach sequence space when equipped with the norm $\lVert \cdot \rVert$ from the previous example. 
\end{example}

\begin{example}\label{ex3}
For each $p\in [1, \infty )$, the set \[ l^p= \bigg{\{} \mathbf{s}=(s_n)_{n\in \Z} \in \mathcal{S}(\Z): \sum_{n\in \Z} \lvert s_n \rvert^p <\infty \bigg{\}} \] is a Banach sequence space when equipped with the norm \[\lVert \mathbf{s} \rVert= \bigg{(}\sum_{n\in \Z} \lvert s_n \rvert^p \bigg{)}^{1/p}.\]
\end{example}

\begin{example}[Orlicz sequence spaces]
Let $\phi \colon (0, +\infty) \to (0, +\infty]$ be a nondecreasing nonconstant left-continuous function. We set $\psi(t)=\int_0^t \phi(s) \, ds$  for $t\ge 0$. Moreover, for each $\mathbf s=(s_n)_{n\in \Z} \in \mathcal S(\Z)$, let $M_\phi (\mathbf s)=\sum_{n\in \Z} \psi(\lvert s_n \rvert)$. Then
\[
B=\bigl\{ \mathbf s \in \mathcal S(\Z) : M_\phi (c \mathbf s)<+\infty \ \text{for some} \ c>0 \bigr\}
\]
is a Banach sequence space when equipped with the norm
\[
\lVert \mathbf s \rVert=\inf \bigl\{ c>0 : M_\phi ( \mathbf s / c ) \le 1 \bigr\}.
\]
\end{example}
\subsection{Important construction}
Let us now introduce sequence spaces that will play important role in our arguments. 
Let $X$ be an arbitrary Banach space and $B$ any Banach sequence space over $\Z$ with norm $\lVert \cdot \rVert_B$. Set
\[
X_B:=\bigg{\{} \mathbf x=(x_n)_{n\in \Z} \subset X: (\lVert x_n\rVert)_{n\in \Z}\in B \bigg{\}}.
\]
Finally, for $\mathbf x=(x_n)_{n\in \Z} \in X_B$ we define 
\begin{equation}\label{nn}
\lVert \mathbf x\rVert_B:=\lVert (\lVert x_n\rVert)_{n\in \Z}\rVert_B.
\end{equation}
\begin{remark}
We emphasize that in~\eqref{nn} we slightly abuse the notation since norms on  $B$ and $X_B$ are denoted in the same way. However, this will cause no confusion since in the rest of the paper we will deal with spaces $X_B$.
\end{remark}
\begin{example}
Let $B=l^\infty$ (see Example~\ref{ex1}). Then, 
\[
X_B=\bigg{\{} \mathbf x=(x_n)_{n\in \Z} \subset X: \sup_{n\in \Z} \lVert x_n\rVert<\infty \bigg{\}}.
\]
\end{example}
The proof of the following result is straightforward (see~\cite{DD, Sasu}).
\begin{proposition}
$(X_B, \lVert \cdot \rVert_B)$ is a Banach space. 
\end{proposition}
\subsection{Exponential dichotomies}\label{sec: exp dichot}
We now recall the notion of an exponential dichotomy which goes back to the landmark work of Perron~\cite{Per30}. 
Let $X=(X, \lVert \cdot \rVert)$ be an arbitrary Banach space and $(A_m)_{m\in \Z}$ a sequence of invertible and bounded linear operators on $X$ such that
\begin{equation}\label{ub}
\sup_{m\in \Z} \lVert A_m\rVert <\infty.
\end{equation}
Set
\[
\cA(m,n)=\begin{cases}
A_{m-1}\cdots A_n & \text{if $m>n$,} \\
\Id & \text{if $m=n$,} \\
A_m^{-1}\cdots A_{n-1}^{-1} & \text{if $m<n$.}
\end{cases}
\]
We say that the sequence $(A_m)_{m\in \Z}$ admits an \emph{exponential dichotomy} if:
\begin{enumerate}
\item there exists a sequence $(P_m)_{m\in \Z}$ of projections on $X$ such that 
\begin{equation}\label{P}
P_{m+1}A_m=A_mP_m \quad \text{for each $m\in \Z$;}
\end{equation}
\item there exist $C, \lambda >0$ such that
\begin{equation}\label{ED1}
\lVert \cA(m,n)P_n\rVert \le Ce^{-\lambda (m-n)} \quad \text{for $m\ge n$}
\end{equation}
and
\begin{equation}\label{ED2}
\lVert \cA(m,n)(\Id-P_n)\rVert \le Ce^{-\lambda (n-m)} \quad \text{for $m\le n$.}
\end{equation}
\end{enumerate}
Let $B$ be an admissible Banach sequence space. We define a bounded linear operator $\mathbb A \colon X_B\to X_B$ by 

\begin{displaymath}
(\mathbb A  \mathbf x)_n=A_{n-1}x_{n-1}, \quad  \text{for $n\in \Z$ and $\mathbf x=(x_n)_{n\in \Z} \in X_B$.}
\end{displaymath}
It follows easily from~\eqref{ub} that $\mathbb A$ is a well-defined and bounded linear operator on $X_B$.

The following result is only a particular case of the results established by Sasu~\cite{Sasu}.
\begin{theorem}\label{t1}
For a sequence $(A_m)_{m\in \Z}$ of invertible and bounded linear operators on $X$ satisfying~\eqref{ub} and an admissible Banach sequence space $B$, the following statements are equivalent:
\begin{enumerate}
\item $(A_m)_{m\in \Z}$ admits an exponential dichotomy;
\item $\Id-\mathbb A$ is an invertible operator on $X_B$.
\end{enumerate}
\end{theorem}
One can use the previous result to establish the following one (see~\cite{Henry2} for example) which tells us that the notion of an exponential dichotomy is robust under small linear perturbations. 

\begin{theorem}\label{rob} 
Let $B$ be an admissible Banach sequence space and  
assume that $(A_m)_{m\in \Z}$ is a sequence of invertible and bounded operators on $X$ that admits an exponential dichotomy and satisfies~\eqref{ub}. Then, there exists $c>0$  such that for any sequence $(B_m)_{m\in \Z}$ of bounded linear operators on $X$ satisfying
\[
\sup_{m\in \Z} \lVert A_m-B_m \rVert \le c,
\]
we have that  $(B_m)_{m\in \Z}$ also admits an exponential dichotomy. Furthermore, there exists $K>0$ such that $\lVert (\Id-\mathbb B)^{-1}\rVert \le K$, where the operator $\mathbb B \colon X_B \to X_B$ is given by
\[
(\mathbb B  \mathbf x)_n=B_{n-1}x_{n-1}, \quad  \text{for $n\in \Z$ and $\mathbf x=(x_n)_{n\in \Z} \in X_B$.}
\]
\end{theorem}

\begin{remark}\label{4:01}
In fact, by carefully  inspecting  the proof of Theorem~\ref{t1} one can  conclude that we can choose constants $C, \lambda >0$ in the notion of an exponential dichotomy uniformly over all sequences $(B_m)_{m\in \Z}$ satisfying the assumptions of Theorem~\ref{rob}. 
\end{remark}
\subsection{A fixed point theorem}
We will also use the following classical consequence of the Banach fixed point theorem (see~\cite{Lanford} for example).
\begin{theorem}\label{FPT}
Let $Z$ be a Banach space and $A$ a differentiable map defined on a neighborhood of $0\in Z$. Furthermore, let $\Gamma$ be a bounded linear operator on $Z$ such that $\Id-\Gamma$ is invertible. Finally, suppose that there exist $\rho>0$ and $\kappa \in (0, 1)$ such that:
\begin{enumerate}
\item for each $z\in Z$ satisfying $\lVert z\rVert \le \rho$,
\[
\lVert (\Id-\Gamma)^{-1}\rVert \cdot \lVert d_zA-\Gamma\rVert \le \kappa;
\]
\item \[
\lVert (\Id-\Gamma)^{-1}\rVert \cdot \lVert A(0)\rVert \le (1-\kappa)\rho.
\]

\end{enumerate}
Then, $A$ has a unique fixed point in $\{z\in Z: \lVert z\rVert \le \rho\}$.
\end{theorem}

\section{Main results}\label{MR}

\subsection{Setup} \label{sec: setup}
Let $B$ be an admissible Banach sequence space,  $X$  a  Banach space and $(A_m)_{m\in \Z}$ a sequence of invertible and bounded linear operators on $X$ that admits an exponential dichotomy and satisfies~\eqref{ub}. Moreover, let $c>0$ be a constant given by Theorem~\ref{rob}.
 Finally, let $f_n \colon X\to X$, $n\in \Z$ be a sequence of differentiable maps such that:
\begin{enumerate}
\item
\begin{equation}\label{fg}
\lVert d_xf_n\rVert \le c\quad \text{for $x\in X$ and $n\in \Z$;}
\end{equation}
\item \begin{equation}\label{UB} \sup_{m\in \Z} \lVert f_m\rVert_{\sup} <\infty,\end{equation}
where $\lVert f_m\rVert_{\sup}:=\sup_{x\in X} \lVert f_m(x)\rVert$;
\item there exist $D>0$ and $r>0$ such that
\begin{equation}\label{der}
\lVert d_xf_n-d_yf_n\rVert \le D\lVert x-y\rVert^r \quad \text{for $x, y\in X$ and $n\in \Z$.}
\end{equation}
\end{enumerate}
We consider a nonautonomous and nonlinear dynamics defined by the equation
\begin{equation}\label{nnd}
x_{n+1}=F_n(x_n),  \quad n\in \Z,
\end{equation}
where \[F_n:=A_n+f_n.\]
We now introduce a notion of a pseudotrajectory associated with the system~\eqref{nnd}. Given $\delta >0$, the sequence $(y_n)_{n\in \Z} \subset X$ is said to be an $(\delta, B)$-\emph{pseudotrajectory} for~\eqref{nnd} if  
$(y_{n+1}-F_n(y_n))_{n\in \Z}\in X_B$ and 
\begin{equation}\label{pseudo}
\lVert (y_{n+1}-F_n(y_n))_{n\in \Z} \rVert_B \le \delta. 
\end{equation}
\begin{remark}
When $B=l^\infty$ (see Example~\ref{ex1}), condition~\eqref{pseudo} reduces to 
\[
\sup_{n\in \Z}  \lVert y_{n+1}-F_n(y_n) \rVert \le \delta.
\]
The above requirement represents a usual definition of a pseudotrajectory in the context of smooth dynamics (see~\cite{Pal00, Pil99}). 
\end{remark}

We say that~\eqref{nnd} has an \emph{$B$-shadowing property} if for every $\varepsilon>0$ there exists $\delta >0$ so that for every $(\delta, B)$-pseudotrajectory $(y_n)_{n\in \Z}$, there exists a sequence $(x_n)_{n\in \Z}$ satisfying \eqref{nnd}  and such that 
$(x_n-y_n)_{n\in \Z} \in X_B$ together with
\begin{equation}\label{wv}\lVert (x_n-y_n)_{n\in \Z}\rVert_B \le \varepsilon.\end{equation}
Moreover, if there exists $L>0$ such that $\delta$ can be chosen as $\delta=L\epsilon$, we say that~\eqref{nnd} has the \emph{$B$-Lipschitz shadowing property}.
\begin{remark}
In the case when $B=l^\infty$ the above definition of shadowing is the classical one~\cite{Pal00, Pil99}. On the other hand, if $B=c_0$ (see Example~\ref{ex2}) we speak about \emph{limit shadowing}, while in the case when $B=l^p$ (see Example~\ref{ex3}) we speak about \emph{$l^p$-shadowing}. 
Hence, our approach offers a unified treatment of various concepts of shadowing. 
\end{remark}

Our main objective is to show that under above assumptions, \eqref{nnd} has the  $B$- Lipschitz shadowing property. We begin with some simple  auxiliary results.
\subsection{Lemmata} \label{sec: lemmata}
\begin{lemma}
Let $\delta >0$ and assume that   $(y_n)_{n\in \Z}$ is  an $(\delta, B)$-pseudotrajectory for~\eqref{nnd}.
Furthermore, we define a map $\mathbf A \colon X_B \to X_B$ given by
\begin{displaymath}
(\mathbf A(\mathbf x))_n=F_{n-1}(y_{n-1}+x_{n-1})-y_n, 
\end{displaymath}
for $n\in \Z$ and $\mathbf x=(x_n)_{n\in \Z} \in X_B$. Then, we have that $\mathbf A$ is a well-defined map. 
\end{lemma}

\begin{proof}
Observe that 
\[
\begin{split}
(\A(\mathbf x))_n &=F_{n-1}(y_{n-1}+x_{n-1})-y_n \\
&=F_{n-1}(y_{n-1}+x_{n-1})-F_{n-1}(y_{n-1})+F_{n-1}(y_{n-1})-y_n \\
&=A_{n-1}x_{n-1}+f_{n-1}(y_{n-1}+x_{n-1})-f_{n-1}(y_{n-1}) \\
&\phantom{=}+F_{n-1}(y_{n-1})-y_n,
\end{split}
\]
for every $n\in \Z$. Then, it follows from~\eqref{fg} that 
\[
\lVert (\A(\mathbf x))_n \rVert \le \big{(}\sup_{m\in \Z} \lVert A_m\rVert \big{)}\cdot \lVert x_{n-1}\rVert +c\lVert x_{n-1}\rVert+\lVert F_{n-1}(y_{n-1})-y_n\rVert,
\]
for each $n\in \Z$ and $\mathbf x=(x_n)_{n\in \Z} \in X_B$. In a view of~\eqref{ub} and~\eqref{pseudo}, we  conclude  that $\A(\mathbf x)\in X_B$. 
\end{proof}

\begin{lemma}\label{L} 
The  map $\A \colon X_B \to X_B$ is differentiable and 
\[
(d_{\mathbf x}\A\mathbf{\boldsymbol{\xi}})_n=A_{n-1}\xi_{n-1}+d_{x_{n-1}+y_{n-1}}f_{n-1}\xi_{n-1},
\]
for $\mathbf x=(x_n)_{n\in \Z}$, $\boldsymbol{\xi}=(\xi_n)_{n\in \Z} \in X_B$.
\end{lemma}

\begin{proof}
Let us fix $\mathbf x=(x_n)_{n\in \Z} \in X_B$ and define a linear operator $L\colon X_B \to X_B$ by
\[
(L\boldsymbol{\xi})_n=A_{n-1}\xi_{n-1}+d_{x_{n-1}+y_{n-1}}f_{n-1}\xi_{n-1}, \quad \text{for $\boldsymbol{\xi}=(\xi_n)_{n\in \Z}\in X_B$.}
\]
It follows readily from~\eqref{ub} and~\eqref{fg} that $L$ is a well-defined and bounded operator.  Moreover, for $\mathbf{h}=(h_n)_{n\in \Z}\in X_B$ we have that
\[
\begin{split}
(\A(\mathbf x+\mathbf h)-\A(\mathbf x)-L\mathbf h)_n &=f_{n-1}(x_{n-1}+y_{n-1}+h_{n-1})-f_{n-1}(x_{n-1}+y_{n-1}) \\
&\phantom{=}-d_{x_{n-1}+y_{n-1}}f_{n-1}h_{n-1} \\
&=\int_0^1 d_{x_{n-1}+y_{n-1}+th_{n-1}}f_{n-1}h_{n-1} \, dt \\
&\phantom{=}-\int_0^1 d_{x_{n-1}+y_{n-1}}f_{n-1}h_{n-1}\, dt.
\end{split}
\]
By~\eqref{der}, 
\begin{align*}
& \lVert (\A(\mathbf x+\mathbf h)-\A(\mathbf x)-L\mathbf h)_n \rVert  \displaybreak[0]\\
&\le \int_0^1 \lVert d_{x_{n-1}+y_{n-1}+th_{n-1}}f_{n-1}h_{n-1}-d_{x_{n-1}+y_{n-1}}f_{n-1}h_{n-1} \rVert \, dt \displaybreak[0]\\
&\le \frac{D}{1+r}\lVert h_{n-1}\rVert^{1+r}.
\end{align*}
Noting that $\lVert h_n \rVert \le \lVert \mathbf h\rVert_B$, we have that
\[
\lVert (\A(\mathbf x+\mathbf h)-\A(\mathbf x)-L\mathbf h)_n \rVert \le \frac{D}{1+r} \lVert \mathbf h\rVert_B^r \cdot \lVert h_{n-1}\rVert,
\]
and consequently 
\[
\lVert \A(\mathbf x+\mathbf h)-\A(\mathbf x)-L\mathbf h \rVert_B \le \frac{D}{1+r} \lVert \mathbf h\rVert_B^{r+1}.
\]
Hence, 
\[
\lim_{\mathbf h\to 0}\frac{\lVert \A(\mathbf x+\mathbf h)-\A(\mathbf x)-L\mathbf h\rVert_{B}}{\lVert \mathbf h\rVert_B}=0,
\]
which implies the desired conclusion.
\end{proof}
Set $\Gamma:=d_{\mathbf 0}\A$. It follows from Lemma~\ref{L} that 
\[
(\Gamma  \boldsymbol{\xi})_n=A_{n-1}\xi_{n-1}+d_{y_{n-1}}f_{n-1}\xi_{n-1},
\]
for $\boldsymbol{\xi}=(\xi_n)_{n\in \Z} \in X_B$ and $n\in \Z$.
\begin{lemma}\label{lem: invertibility}
We have that $\Id-\Gamma$ is an invertible operator on $X_B$. Moreover, there exists a constant $K>0$, independent on the pseudotrajectory $\mathbf y$ such that  \[ \lVert (\Id-\Gamma)^{-1}\rVert \leq K.\]
\end{lemma}

\begin{proof}
The statement follows directly from Theorems~\ref{t1} and~\ref{rob} together with~\eqref{fg} (and the choice of $c$). 
\end{proof}

\subsection{Main results} \label{sec: main results}
The following is our  main result. 
\begin{theorem}\label{cor: uniqueness}
The system~\eqref{nnd} has an $B$-Lipschitz shadowing property. Furthermore, if $\varepsilon$ is sufficiently small, the sequence $\mathbf x=(x_n)_{n\in \Z}$ satisfying~\eqref{nnd} and~\eqref{wv} is unique. 
\end{theorem}

\begin{proof}
We wish to apply Theorem~\ref{FPT}. It follows from~\eqref{der} that 
\[
\lVert ((d_{\mathbf z}\A-\Gamma){\boldsymbol \xi})_n\rVert \le D\lVert z_{n-1}\rVert^r \cdot \lVert \xi_{n-1}\rVert \le D \lVert \mathbf z\rVert_B^r \cdot \lVert \xi_{n-1}\rVert , 
\]
and thus
\[
\lVert d_{\mathbf z}\A-\Gamma\rVert \le  D \lVert \mathbf z\rVert_B^r, \quad \text{for $\mathbf z\in X_B$.}
\]
Let $K>0$ be such that $\lVert (\Id-\Gamma)^{-1}\rVert \le K$. As we noted in Lemma \ref{lem: invertibility}, we can choose $K$ independently on $\mathbf y$. 

Take now $\varepsilon>0$ such that $DK\varepsilon^r \le 1/2$ and set $L=1/4K$. Thus, for  $\delta=L\varepsilon>0$ we have that $K\delta \le \varepsilon/2$. We conclude that the assumptions of Theorem~\ref{FPT} are satisfied with $\kappa=1/2$ and $\rho=\varepsilon$ and consequently $\A$ has a unique fixed point $\mathbf z\in X_B$ such that $\lVert \mathbf z\rVert_B\le \varepsilon$. By setting $\mathbf x:=\mathbf y+\mathbf z$, we obtain the desired conclusions. 
\end{proof}

We have the following simple consequence of Theorem~\ref{cor: uniqueness}.
\begin{corollary}[Expansivity] \label{cor: expansivity}
There exists $\varepsilon>0$ so that if $\mathbf x=(x_n)_{n\in \Z}$ and $\mathbf z=(z_n)_{n\in \Z}$ are sequences satisfying \eqref{nnd} and 
\begin{displaymath}
\lVert (x_n-z_n)_{n\in \Z} \rVert_B \leq \varepsilon
\end{displaymath}
then, $\mathbf x=\mathbf z$. 
\end{corollary}

\begin{proof}
Choose $\varepsilon>0$  small enough  as in the statement of Theorem~\ref{cor: uniqueness} and corresponding $\delta >0$ as in the definition of the $B$-Lipschitz shadowing property. Obviously,  $\mathbf z$ is a $(\delta,B)$-pseudotrajectory for~\eqref{nnd} which is $(\epsilon, B)$-shadowed by itself and $\mathbf x$. Hence, the uniqueness part in Theorem~\ref{cor: uniqueness} implies that $\mathbf x=\mathbf z$ as claimed.
\end{proof}

Under certain periodicity assumption for system~\eqref{nnd}, we
can  also formulate a version of the Anosov closing lemma in our setting.

\begin{corollary}
Assume that there exists $N\in \mathbb N$ such that $F_{n+N}=F_n$ for each $n\in \mathbb N$. Then, there exists $L>0$ such that for any $\varepsilon >0$ sufficiently small and for every $(L\varepsilon, l^\infty)$-pseudotrajectory $\mathbf y=(y_n)_{n\in \Z}$ for~\eqref{nnd} that satisfies $y_n=y_{n+N}$ for $n\in \mathbb N$, there exists 
a solution $\mathbf x=(x_n)_{n\in \Z}$ of~\eqref{nnd} such that $\sup_{n\in \Z} \lVert x_n-y_n\rVert \le \varepsilon$ and $x_n=x_{n+N}$ for $n\in \mathbb N$.

\end{corollary}

\begin{proof}
Applying Theorem~\ref{cor: uniqueness} (for $B=l^\infty$) we obtain the existence of a sequence $\mathbf x=(x_n)_{n\in \Z}$ solving~\eqref{nnd} such that $\sup_{n\in \Z} \lVert x_n-y_n\rVert \le \varepsilon$. Let us define a new sequence $\mathbf x'=(x_n')_{n\in \Z} \subset X$ by
\[
x_n'=x_{n+N} \quad \text{for $n\in \Z$.}
\]
It is easy to verify that $\mathbf x'$ solves~\eqref{nnd} and it obviously satisfies $\sup_{n\in \Z} \lVert x_n'-y_n\rVert \le \varepsilon$. Hence, the uniqueness in Theorem~\ref{cor: uniqueness} implies that $\mathbf x=\mathbf x'$ which immediately yields the desired conclusion.
\end{proof}

\subsection{Shadowing of linear systems}
In this subsection we deal with the system~\eqref{nnd} in the particular case when  $f_n=0$ for each $n\in \Z$, i.e. when $F_n=A_n$ for every $n\in \Z$. Hence, we deal with linear dynamics
\begin{equation}\label{ld}
x_{n+1}=A_n x_n, \quad n\in \Z. 
\end{equation}
\begin{corollary}\label{735}
The system~\eqref{ld} has an $B$-Lipschitz shadowing property. Furthermore, for each $\varepsilon >0$, the sequence $\mathbf x=(x_n)_{n\in \Z}$ satisfying~\eqref{wv} and~\eqref{ld} is unique. 
\end{corollary}

\begin{proof}
In a view of Theorem~\ref{cor: uniqueness}, it only remains to establish the uniqueness of $\mathbf x$. Assume that $\mathbf x=(x_n)_{n\in \Z}$ and $\mathbf x'=(x_n')_{n\in \Z}$ satisfy~\eqref{ld}, $\lVert \mathbf x-\mathbf y\rVert_B \le \varepsilon$ and $\lVert \mathbf x'-\mathbf y\rVert_B \le \varepsilon$. Hence,
$\lVert \mathbf x-\mathbf x'\rVert \le 2\varepsilon$. 

On the other hand, since $\mathbf x$ and $\mathbf x'$ satisfy~\eqref{ld}, we have that \[(\Id-\mathbb A)(\mathbf x-\mathbf x')=0\] and thus Theorem~\ref{t1} implies that $\mathbf x=\mathbf x'$.
\end{proof}

\begin{corollary}\label{755}
For each sequence $\mathbf y=(y_n)_{n\in \Z} \subset X$ such that \[\lim_{\lvert n\rvert \to \infty}\lVert y_{n+1}-A_ny_n\rVert=0,\] there exists a unique sequence $\mathbf x=(x_n)_{n\in \Z}$ that solves~\eqref{ld} and with the property that
\begin{equation}\label{736}
\lim_{\lvert n\rvert\to \infty} \lVert x_n-y_n\rVert=0.
\end{equation}
\end{corollary}

\begin{proof}
Take $B=c_0$ (see Example~\ref{ex2}).
We note that $\mathbf y$ is an $(L\varepsilon, B)$-pseudotrajectory for~\eqref{ld} for some $\varepsilon >0$ (where $L>0$ comes from the definition of $B$-Lipschitz shadowing).  By Corollary~\ref{735}, we can find a sequence $\mathbf x=(x_n)_{n\in \Z}$ that solves~\eqref{ld} and satisfies
$(x_n-y_n)_{n\in \Z} \in B$ (in fact, we also have that $\lVert (x_n-y_n)_{n\in \Z} \rVert_B=\sup_{n\in \Z} \lVert x_n-y_n\rVert \le \varepsilon$), which immediately yields~\eqref{736}.

The uniqueness of $\mathbf x$ can be established by using Theorem~\ref{t1}, as in the proof of Corollary~\ref{735}.
\end{proof}

\begin{corollary}\label{756}
Take $1\le p<\infty$. For each sequence $\mathbf y=(y_n)_{n\in \Z} \subset X$ such that \[\sum_{n\in \Z}\lVert y_{n+1}-A_ny_n\rVert^p <\infty,\] there exists a unique sequence $\mathbf x=(x_n)_{n\in \Z}$ that solves~\eqref{ld} and with the property that
\[
\sum_{n\in \Z}\lVert x_n-y_n\rVert^p <\infty. 
\]
\end{corollary}

\begin{proof}
The proof can be obtain by repeating the arguments in the proof of Corollary~\ref{755}  and by using $B=l^p$ (see Example~\ref{ex3}) instead of $B=c_0$. 
\end{proof}

\begin{remark}
In a very  particular case when~\eqref{ld} is an autonomous system, i.e. $(A_n)_{n\in \Z}$ is a constant sequence of operators, Corollary~\ref{735} applied to $B=l^\infty$ together with  Corollaries~\ref{755} and~\ref{756} gives us~\cite[Theorem A.]{BCDMP}.
\end{remark}

\section{Applications}\label{A}
In this section we present some applications of our main results. 
\subsection{A nonautonomous version of the  Grobman-Hartman theorem}
Let $(A_m)_{m\in \Z}$ be  a sequence of bounded linear operators on $X$ as in Subsection~\ref{sec: setup}.  Let $c>0$ be given by Theorem \ref{rob} and fix $D,r>0$. Associated to these parameters by Theorem \ref{cor: uniqueness} (applied to $B=l^\infty$), consider $\varepsilon>0$ sufficiently small such that $D(3\varepsilon)^r <c/2$ and $\delta=L\varepsilon>0$. Moreover, suppose $\varepsilon$ is so small that $3\varepsilon$ still satisfies Theorem \ref{cor: uniqueness}. Let $(g_n)_{n\in \Z}$ be a sequence of maps $g_n\colon X\to X$ satisfying \eqref{fg}  with $c/2$ instead of $c$ and \eqref{der} and such that
\[\lVert g_n\rVert_{\sup} \le \delta \quad  \text{for each $n\in \Z$.} \] We consider a  difference equation
\begin{equation}\label{j}
y_{n+1}=G_n(y_n) \quad n\in \Z,
\end{equation}
where $G_n:=A_n+g_n$. By decreasing $\delta$ (if necessary), we have that $G_n$ is a homeomorphism for each $n\in \Z$ (see~\cite{BV1}). We define
\[
\mathcal{G}(m,n)=
\begin{cases}
G_{m-1}\circ \ldots \circ G_n & \text{if $m>n$,}\\
\Id & \text{if $m=n$,}\\
G_m^{-1}\circ \ldots \circ G_{n-1}^{-1} & \text{if $m<n$.}
\end{cases}
\]
\begin{theorem}\label{NGH}
There exists a unique sequence $h_m \colon X \to X$, $m\in \Z$ of homeomorphisms such that  for each $m\in \Z$,
\begin{equation}\label{644}
h_{m+1}\circ G_m=A_m \circ h_m
\end{equation}
and
\begin{equation}\label{702}
\lVert h_m-\Id\rVert_{\sup}=\sup_{x\in X} \lVert h_m(x)-x\rVert \le \epsilon. 
\end{equation}
\end{theorem}

\begin{proof}
Fix $m\in \Z$ and $y\in X$ and define a sequence $\mathbf y=(y_n)_{n\in \Z}$ by $y_n=\mathcal{G}(n, m)y$ for $n\in \Z$. Note that $\mathbf y$ is a solution of~\eqref{j}. Then,
\[
\sup_{n\in \Z}\lVert y_{n+1}-A_n y_n\rVert=\sup_{n\in \Z} \lVert g_n(y_n)\rVert \le \delta.
\]
Hence, it follows from Corollary~\ref{735} (applied to the case when $B=l^\infty$) that there exists a unique sequence $\mathbf x=(x_n)_{n\in \Z}$ such that
$x_{n+1}=A_nx_n$ for $n\in \Z$ and $\sup_{n\in \Z}\lVert x_n-y_n \rVert \le \varepsilon$. Set
\[
 h_m(y)=h_m(y_m):=x_m. 
\]
It is easy to verify that~\eqref{644} holds.  Furthermore, 
\[
\lVert h_m(y)-y\rVert=\lVert x_m-y_m\rVert \le \epsilon, 
\]
which yields~\eqref{702}. We will now prove that each $h_m$ is a homeomorphism.  Let us start with the following simple auxiliary result. 
\begin{lemma}\label{lem: unif expansive}
Let $\mathbf x=(x_n)_{n\in \Z}$ and $ \tilde{\mathbf x }=(\tilde{x}_n)_{n\in \Z}$ be two  sequences such that \[x_{n+1}=A_nx_n\quad  \text{and} \quad \tilde{x}_{n+1}=A_n\tilde{x}_n, \]for every $n\in \Z$. Then, for every $\rho >0$ there exists $N\in \mathbb N$ such that if $\lVert x_n -\tilde{x}_n\rVert \le 3\varepsilon$ for every $| n| \leq N$, then we have that $\lVert x_0-\tilde{x}_0\rVert \le \rho$. In addition, analogous property holds for solutions of~\eqref{j}.
\end{lemma}

\begin{proof}[Proof of the lemma]
It follows from~\eqref{P} and~\eqref{ED1} that 
\[
\lVert P_0(x_0-\tilde x_0)\rVert =\lVert \cA(0, -N)P_{-N}(x_{-N}-\tilde x_{-N})\rVert \le Ce^{-\lambda N}\lVert x_{-N} -\tilde x_{-N}\rVert.
\]
Similarly, it follows from~\eqref{P} and~\eqref{ED2} that 
\[
\lVert (\Id-P_0)(x_0-\tilde x_0)\rVert=\lVert \cA(0, N)(\Id-P_N)(x_N-\tilde x_N)\rVert \le Ce^{-\lambda N} \lVert x_N-\tilde x_N\rVert. 
\]
Hence,
\[
\lVert x_0-\tilde x_0\rVert \le Ce^{-\lambda N} (\lVert x_{-N} -\tilde x_{-N}\rVert+ \lVert x_N-\tilde x_N\rVert),
\]
and therefore choosing $N$ such that $6Ce^{-\lambda N}\varepsilon<\rho$ we obtain desired conclusion.  As for the nonlinear case, let $\mathbf y=(y_n)_{n\in \Z}$ and $\tilde{\mathbf y }=(\tilde{y}_n)_{n\in \Z}$ be sequences satisfying \eqref{j} and suppose $\lVert y_n -\tilde{y}_n\rVert \le 3\varepsilon$ for every $| n| \leq N$. Setting $z_n=y_n -\tilde{y}_n$ we get that 
\begin{displaymath}
z_{n+1}=G_n(y_n)-G_n(\tilde{y}_n)=d_{\tilde{y}_n}G_nz_n + G_n(y_n)-G_n(\tilde{y}_n) - d_{\tilde{y}_n}G_nz_n.
\end{displaymath}
Thus, 
\begin{displaymath}
z_{n+1}=(L_n+T_n)z_n
\end{displaymath}
where $L_n:=d_{\tilde{y}_n}G_n =A_n+d_{\tilde{y}_n}g_n $ and
\begin{displaymath}
T_n:=\int_0^1 \left(d_{ty_n+(1-t)\tilde{y}_n}G_n -d_{\tilde{y}_n}G_n\right) dt.
\end{displaymath}
Now, using \eqref{der} we get that $\|T_n\|\leq D\|y_n -\tilde{y}_n\|^r\leq D(3\varepsilon)^r<c/2$ for every $| n| \leq N$. Therefore, it follows from Theorem \ref{rob} that the sequence $(B_n)_{n\in \Z}$ given by $B_n=L_n+T_n$ for $| n| \leq N$ and $B_n=L_n$ for $|n|>N$ admits an exponential dichotomy with constants $C$ and $\lambda$ as in Subsection \ref{sec: exp dichot} depending only on $(A_n)_{n\in \Z}$ and $c$ (see Remark~\ref{4:01}).  Thus, choosing $N$ such that $6Ce^{-\lambda N}\varepsilon<\rho$ and proceeding as in the linear case we get the desired result. 
\end{proof}

Let us now establish continuity of $h_0$ (the same argument applies for every $h_m$). 
For $\rho>0$, take  $N\in \mathbb N$  given by Lemma~\ref{lem: unif expansive}. By continuity of maps $g_n$, there exists $\eta >0$ such that for every $y,z\in X$ which satisfy $\|y-z\|<\eta$, we have that
\begin{displaymath}
\| \mathcal{G}(n,0)y -\mathcal{G}(n,0)z  \| <\varepsilon,
\end{displaymath}
for every  $| n| \leq N$. 

Take $y,z\in X$ satisfying $\|y-z\|<\eta$ and consider $y_n=\mathcal{G}(n,0)y$, $z_n=\mathcal{G}(n,0)z$,  $x_n=\mathcal{A}(n,0) h_0(y)$ and $\tilde{x}_n=\mathcal{A}(n,0)h_0(z)$ for  $n\in \Z$. For every $n\in \Z$ satisfying $|n|\leq N$, we have that 
\begin{displaymath}
\begin{split}
\| x_n-\tilde{x}_n  \|&=\|  \mathcal{A}(n,0)h_0(y)-\mathcal{A}(n,0)h_0(z) \| \\
&\leq \| h_{n}( \mathcal{G}(n,0)y) - h_{n}( \mathcal{G}(n,0)z)  \| \\
& \leq \| h_{n}( \mathcal{G}(n,0)y) -\mathcal{G}(n,0)y\| + \|\mathcal{G}(n,0)y-\mathcal{G}(n,0)z\|\\ 
&+\|\mathcal{G}(n,0)z-  h_{n}(\mathcal{G}(n,0)z)  \| \\
& =\|h_{n}(y_{n})-y_{n}\| +\|\mathcal{G}(n,0)y-\mathcal{G}(n,0)z\| \\
&+\| z_{n}-h_{n}(z_n)\| \\
&\le \varepsilon + \varepsilon +\varepsilon =3\varepsilon.
\end{split}
\end{displaymath}
Thus, since $(x_n)_{n\in \Z}$ and $(\tilde{x}_n)_{n\in \Z}$ are solutions of $x_{n+1}=A_nx_n$, $n\in \Z$, it follows from Lemma \ref{lem: unif expansive} that $\|h_0(y)-h_0(z)\|\le \rho$ proving that $h_0$ is continuous.

We now prove that $h_m$ is injective for each $m\in \Z$. Suppose that  there exist $y,z\in X$ such that $h_m(y)=h_m(z)$. We define sequences $(y_n)_{n\in \Z}$, $(z_n)_{n\in \Z}$ and $(x_n)_{n\in \Z}$ by $y_n=\mathcal{G}(n,m)y$, $ z_n= \mathcal{G}(n,m)y$ and $x_n=\mathcal{A}(n,m) h_m(y)=\mathcal{A}(n,m)h_m(z)$, $n\in \Z$. Then, by the definition of $h_m$ we have that
\begin{displaymath}
\sup_{n\in \Z}\lVert x_n-y_n \rVert \le \varepsilon \text{ and } \sup_{n\in \Z}\lVert x_n-z_n \rVert \le \varepsilon.
\end{displaymath}
In particular,
\begin{displaymath}
\sup_{n\in \Z}\lVert y_n-z_n \rVert \le 2\varepsilon.
\end{displaymath}
Then,  Corollary \ref{cor: expansivity} (applied for $B=l^\infty$)  implies that $y_n=z_n$ for every $n\in \Z$. Consequently,  $y=y_m=z_m=z$ and thus $h_m$ is injective. 

Let us now establish surjectivity of $h_m$. Take $x\in X$ and consider a sequence $x_n=\mathcal{A}(n,m)x$. Then,
\begin{displaymath}
\sup_{n\in \Z} \| x_{n+1}-G_nx_n\|=\sup_{n\in \Z}\| A_nx_n- G_nx_n \|=\sup_{n\in \Z}\|g_n(x_n)\|\le \delta.
\end{displaymath}
Hence, $(x_n)_{n\in \Z}$ is a $(\delta, l^\infty)$-pseudotrajectory for \eqref{j}. In particular, by Theorem \ref{cor: uniqueness} applied to $(G_n)_{n\in \Z}$ there exists a unique sequence $(y_n)_{n\in \Z}$ satisfying \eqref{j} and such that $\sup_{n\in \Z}\lVert x_n-y_n \rVert \le \varepsilon$. Therefore, $h_m(y_m)=x_m=x$ proving that $h_m$ is surjective.

The proof of continuity of $h_m^{-1}$ is completely analogous to the proof of continuity of $h_m$ and therefore we omit it. 

Finally, it remains to establish uniqueness of $h_m$. Indeed,  suppose that $(\tilde{h}_m)_{m\in \Z}$ is  a sequence satisfying~\eqref{644} and~\eqref{702}. Take $y\in X$, $m\in \Z$  and consider  a sequence $y_n=\mathcal{G}(n,m)y$, $n\in \Z$. Then, 
\begin{displaymath}
h_{n+1}(y_{n+1})=A_n h_n(y_n)  \text{ and } \tilde{h}_{n+1}(y_{n+1})=A_n \tilde{h}_n (y_n).
\end{displaymath}
Moreover, $\|h_n(y_n)-y_n\|\le \varepsilon$ and $\|\tilde{h}_n(y_n)-y_n\|\le \varepsilon$ for every $n\in \Z$. Thus, by the uniqueness in Theorem \ref{cor: uniqueness} it follows that $h_n(y_n)=\tilde{h}_n(y_n)$ for every $n\in \Z$. Consequently, $h_m(y)=\tilde h_m(y)$.  Since $y$ was arbitrary we conclude that $h_m=\tilde h_m$ and the proof of the theorem is completed. 
\end{proof}

\begin{remark}\label{10:53}
Theorem~\ref{NGH} can be described  as a nonautonomous version of the classical Grobman-Hartman theorem~\cite{Hart60}. The first result of this type has been established by Palmer~\cite{Palmer}  for finite-dimensional dynamics with continuous time. Subsequent generalizations for infinite-dimensional  dynamics  that admits a nonuniform exponential dichotomy are due to Barreira
and Valls~\cite{BV1, BarVal-DCDS07} (see also~\cite{BDV1} for simple proof). 

When compared with main results in~\cite{BV1,  BarVal-DCDS07, BDV1}, our Theorem~\ref{NGH} works under stronger assumptions that maps $g_n$ are differentiable and that~\eqref{der} holds. 

However, our goal was not to refine results from those papers (which in fact seem to be rather optimal)  but rather to offer a new approach to the problem of linearization of a nonautonomous dynamics based on the shadowing theory we developed in previous section. 
\end{remark}

\subsection{Preservation of positive Lyapunov exponents}
Using our main results, we can also formulate conditions under which positive Lyapunov exponents associated with a linear dynamics remain unchanged under small nonlinear perturbations.  For more general results related to the preservation of Lyapunov exponents under perturbations we refer to~\cite[Section 7.]{BDV2} and references therein. 
We continue to use the same notation as in the previous subsection.  
\begin{corollary}
Assume that $(y_n)_{n\in \Z}$ is a solution of~\eqref{j} such that
\[
\lambda:=\limsup_{n\to \infty} \frac 1 n \log \lVert y_n\rVert>0. 
\]
Then, there exists a solution $(x_n)_{n\in \Z}$ of~\eqref{ld} such that
\begin{equation}\label{800}
\lambda=\limsup_{n\to \infty} \frac 1 n \log \lVert x_n\rVert.
\end{equation}
\end{corollary}

\begin{proof}
Observe that 
\[
\sup_{n\in \Z}\lVert y_{n+1}-A_n y_n\rVert=\sup_{n\in \Z} \lVert g_n(y_n)\rVert \le \delta.
\]
Hence, it follows from Corollary~\ref{735} (applied to the case when $B=l^\infty$) that there exists a unique sequence $\mathbf x=(x_n)_{n\in \Z}$ such that
$x_{n+1}=A_nx_n$ for $n\in \Z$ and $\sup_{n\in \Z}\lVert x_n-y_n \rVert \le \varepsilon$. It now follows readily that~\eqref{800} holds. 
\end{proof}

%%%%%%%%%%%%%%%%%%%%%%%%%%%%%%%%%%%%%%%%%%%%%%%%%%%%%%%%%%%%%%%%%%%
\medskip{\bf Acknowledgements.} L. B. was partially supported by a CAPES-Brazil postdoctoral fellowship under Grant No. 88881.120218/2016-01 at the University of Chicago. D. D. was supported by the Croatian Science Foundation under the project IP-2014-09-2285.

%%%%%%%%%%%%%%%%%%%%%%%%%%%%%%%%%%%%%%%%%%%%%%%%%%%%%%%%%%%%%%%%%%%%%%%%%

\end{document}